\documentclass[12pt]{article}

\usepackage[utf8]{inputenc}
\usepackage[T1]{fontenc}
\usepackage{amssymb,latexsym}
\usepackage{amsmath}
\usepackage{amsthm}
\usepackage{thmtools}
\usepackage{graphicx}
\usepackage{xcolor}
\usepackage{yfonts}
\usepackage[shortlabels]{enumitem}
\usepackage{cleveref}

\usepackage{listings}%
\usepackage{esvect} 

\usepackage{subcaption}
\usepackage{url}

\usepackage{diagbox}

\def\N{\mathbb N}

\def\uu{\mathbf{u}}
\def\vv{\mathbf{v}}

\def\Lang{\mathcal{L}}
\def\Bext{\mathrm{Bext}}
\def\Lext{\mathrm{Lext}}
\def\Rext{\mathrm{Rext}}

 \newtheorem{thm}{Theorem}

 \newtheorem{coro}[thm]{Corollary}
 \newtheorem{corollary}[thm]{Corollary}
 \newtheorem{conjecture}[thm]{Conjecture}
 \newtheorem{lemma}[thm]{Lemma}

 \newtheorem{proposition}[thm]{Proposition}
 
 \newtheorem{defi}[thm]{Definition}
 \newtheorem{rem}[thm]{Remark}
 \newtheorem{example}[thm]{Example}

 \renewcommand{\ge}{\geqslant}

\title{Generalized reflection complexity}
\author{Adam Bla\v zek \and Lubom\'ira Dvo\v r\'akov\'a
\and Edita Pelantov\'a}

\begin{document}

\maketitle

\begin{abstract}

The reflection complexity of a~sequence, introduced recently by Allouche, Campbell, Li, Shallit, and Stipulanti, counts the factors of a~given length up to reversal.  
We generalize this notion considering a~group $G$ generated by antimorphisms over a given alphabet $\mathcal A$. 
The $G$-complexity of a~sequence over $\mathcal A$ counts the factors of a~given length up to the action of $G$. 
For $G$ generated by the reflection alone, the $G$-complexity coincides with the reflection complexity.
We characterize eventual periodicity in terms of the $G$-complexity. 
Then we  focus on  the binary alphabet and  the group $H$ generated by reflection and the antimorphism exchanging letters.
We provide a~simple formula for the $H$-complexity in the case of $H$-rich sequences and determine the precise values for complementary symmetric Rote sequences and generalized binary Thue-Morse sequences.
Finally, we describe the aperiodic recurrent sequences having the minimum $H$-complexity -- they are exactly the complementary symmetric Rote sequences, 
while the binary sequences with the minimum reflection complexity correspond to the Sturmian sequences. 

\end{abstract}

\section{Introduction}
The {\it factor complexity} $\mathcal{C}(n)$ of a~sequence $\mathbf u$
counts the distinct factors of $\mathbf u$ of length $n$. It is one of the
basic tools for measuring the ``disorder'' of a~sequence. The classical
theorem of Morse and Hedlund~\cite{MorseHedlund1938} states that $\mathbf u$ is
eventually periodic if and only if $\mathcal{C}(n+1)=\mathcal{C}(n)$ for
some $n\in\mathbb N$. Hence every aperiodic sequence satisfies
$\mathcal{C}(n)\geq n+1$ for all $n$. The aperiodic sequences
attaining this bound are called Sturmian sequences. 

The relation between periodicity and complexity and the characterization of aperiodic sequences with the minimum complexity may be generalized to complexity functions other than the factor complexity.

Recently, Allouche, Campbell, Li, Shallit, and Stipulanti~\cite{AlCaLiShSt2025}
introduced the {\it reflection complexity} $r$ of a sequence. It counts the factors of length $n$ up to reversal, so a factor $w$ and its mirror image
$R(w)$ are counted once. Among many other results, they computed the reflection complexity for several classical
families, e.g., Sturmian sequences, complementary symmetric (CS) Rote sequences~\footnote{sequences with factor complexity ${\mathcal C}(n)=2n$ for all $n\in \mathbb N,\ n\geq 1$, whose language is closed under the exchange of letters} and the Thue--Morse sequence. Moreover, they characterized Sturmian sequences by the reflection complexity. They also showed that, unlike the factor complexity,
the reflection complexity need not be non-decreasing. 
Dvořáková and Pelantová~\cite{DvPe2026}
characterized eventually periodic sequences by the condition
$r(n+2)=r(n)$ for some $n$, which is an analogue of the Morse--Hedlund
theorem. They also described all sequences satisfying $r(n+2)=r(n)+1$ for
every $n$. 
Besides Sturmian sequences, these include images of Sturmian
sequences under a~simple morphism over a~ternary alphabet.

The reflection complexity is one instance of a broader idea: counting factors
up to a group action. Charlier, Puzynina, and Zamboni~\cite{CPZ} introduced
the {\it group complexity}. For a sequence
$\omega=(G_n)_{n\geq 1}$ of subgroups of the symmetric groups $S_n$, the
group complexity $p_\omega(n)$ counts the factors of length $n$ up to
permutations of their positions by elements of $G_n$. This notion
interpolates between the factor complexity (all $G_n$ trivial) and the abelian
complexity (all $G_n=S_n$). Taking $G_n=\langle \rho_n \rangle$, where
$\rho_n$ reverses the order of positions, gives exactly the reflection
complexity. 

For the sake of completeness, let us add that the formal concept and the notion ``abelian complexity'' itself for general words were introduced and systematized  by Richomme, Saari,  and Zamboni~\cite{RiSaZa2010}
 and~\cite{RiSaZa2011}, although the idea had already appeared implicitly in the work of  Coven and  Hedlund \cite{CovenHedlund1973}.

Launer, Puzynina, and Voloshinova~\cite{Pu2026} recently studied
which values the group complexity can take. For every length $n$, they ask
whether each integer between the abelian and the factor complexity is
attained for a suitable subgroup of $S_n$. They proved that this universal group complexity property holds for
every Sturmian word.

In this paper, we generalize the notion of  reflection complexity  in the following way: for a given group $G$ generated by  
antimorphisms over an alphabet $\mathcal{A}$, the $G$-complexity $\mathcal{C}_G(n)$ 
counts the
factors of length $n$ up to the action of $G$. If  $G$ is generated by the reflection $R$ only, then  $\mathcal{C}_G$
coincides with the reflection complexity. As soon as $G$
contains an antimorphism acting non-trivially on letters, $\mathcal C_G$ is 
no longer a group complexity in the sense of Charlier, Puzynina, and Zamboni~\cite{CPZ}.  

The paper is organized as follows. We first derive properties of the $G$-complexity that hold for an arbitrary
group $G$. In particular, we show that  
$\mathcal C_G(n+2)\geq \mathcal C_G(n)$ for all $n$ (Proposition \ref{prop:Neklesa}).  Consequently, a~sequence  is eventually periodic if and only if 
$\mathcal C_G(n+2)=\mathcal C_G(n)$ for all sufficiently large $n$ (Corollary \ref{cor:periodicity}). Then we
define a condition on the group $G$, called Property~$\textswab{P}$, and prove  that  under Property~$\textswab{P}$, a~sequence is aperiodic if and only if
$\mathcal C_G(n+2)\geq \mathcal C_G(n)+1$ for all $n\geq 1$ (Proposition \ref{pro:neperiodicity}).


The rest of the paper concerns the binary alphabet and the  group $H$ generated by reflection and the antimorphism exchanging letters. The group $H$ satisfies Property~$\textswab{P}$.   For $H$-rich
sequences we obtain a~simple formula that relates $H$-complexity  $\mathcal C_H$ and  the factor complexity $\mathcal C$ (Proposition \ref{pro:CGforGrich}). This formula allows us to 
determine the $H$-complexity for CS Rote sequences
and for generalized binary  Thue--Morse sequences.
Our main result describes the aperiodic recurrent sequences with the slowest possible growth of $\mathcal C_H$: a~recurrent binary sequence
satisfies $\mathcal C_H(n+2)=\mathcal C_H(n)+1$ for every $n$ if and only
if it is a~CS Rote sequence (Theorem~\ref{thm:aperiodicRH}). 
Thus, for the $H$-complexity, CS Rote sequences play the role that Sturmian sequences play
for the factor complexity and reflection complexity.  The final section lists open problems.


\section{Preliminaries}\label{sec:preliminaries}

An \emph{alphabet} $\mathcal{A}$ is a finite set of symbols called \emph{letters}.
A \emph{word} over $\mathcal{A}$ is a finite string $w = w_1 w_2 \cdots w_n$ of letters
$w_i \in \mathcal{A}$; its \emph{length} is $|w| = n$. The \emph{empty word} is denoted
by $\varepsilon$, $|\varepsilon| = 0$. The set of all finite words over $\mathcal{A}$
equipped with concatenation forms the free monoid $\mathcal{A}^*$, and we put
$\mathcal{A}^+ = \mathcal{A}^* \setminus \{\varepsilon\}$. For $w \in \mathcal{A}^*$ and
$a \in \mathcal{A}$, the symbol $|w|_a$ denotes the number of occurrences of the letter $a$ in $w$.

A \emph{sequence} (or an infinite word) over $\mathcal{A}$ is an infinite string
$\uu = u_0 u_1 u_2 \cdots$ with $u_i \in \mathcal{A}$ for all $i \in \mathbb{N}$.
The set of all sequences over $\mathcal{A}$ is denoted by $\mathcal{A}^{\mathbb{N}}$.

A word $w$ is a \emph{factor} of a word or a sequence $v$ if $v = pws$ for some
word $p$ and some word (or sequence) $s$. If $p = \varepsilon$, then $w$ is a
\emph{prefix} of $v$; if $s = \varepsilon$, then $w$ is a \emph{suffix} of $v$.
An index $i \in \mathbb{N}$ is an \emph{occurrence} of a factor $w$ in $\uu$ if
$u_i u_{i+1} \cdots u_{i+|w|-1} = w$.

The \emph{language} $\Lang(\uu)$ of a sequence $\uu$ is the set of all its factors,
and $\Lang_n(\uu)$ denotes the set of factors of $\uu$ of length $n$.
The \emph{factor complexity} of $\uu$ is the mapping ${\mathcal C}: \mathbb{N} \to \mathbb{N}$
given by
\[
  {\mathcal C}(n) = \#\Lang_n(\uu),
\]
and we write $\Delta {\mathcal C}(n) = {\mathcal C}(n+1) - {\mathcal C}(n)$ for its first difference.
Since every factor of $\uu$ can be prolonged to the right, $\mathcal C$ is non-decreasing.

A sequence $\uu$ is \emph{eventually periodic} if $\uu = v z z z \cdots = v z^\omega$
for some $v \in \mathcal{A}^*$ and $z \in \mathcal{A}^+$; if moreover $v = \varepsilon$,
then $\uu$ is \emph{(purely) periodic}. A sequence which is not eventually periodic is called
\emph{aperiodic}.

A sequence $\uu$ is \emph{recurrent} if every factor of $\uu$ occurs in $\uu$
infinitely many times. It is \emph{uniformly recurrent} if for every factor
$w \in \Lang(\uu)$ there exists $N \in \mathbb{N}$ such that every factor of $\uu$ of length
$N$ contains $w$. Let $w$ be a factor of a recurrent sequence $\uu$ and let $i < j$ be two
consecutive occurrences of $w$ in $\uu$. Then the word $u_i u_{i+1} \cdots u_{j-1}$ is called a
\emph{return word} to $w$ in $\uu$. A~recurrent sequence $\uu$ is uniformly recurrent
if and only if every factor of $\uu$ has only finitely many return words.

\subsection{Special factors and complexity}

Let $\uu$ be a sequence over $\mathcal{A}$ and $w \in \Lang(\uu)$.
The sets of \emph{left}, \emph{right}, and \emph{bilateral extensions} of $w$ are
\[
\begin{aligned}
  \Lext(w) &= \{ aw : a \in \mathcal{A} \text{ and } aw \in \Lang(\uu) \},\\
  \Rext(w) &= \{ wb : b \in \mathcal{A} \text{ and } wb \in \Lang(\uu) \},\\
  \Bext(w) &= \{ awb : a, b \in \mathcal{A} \text{ and } awb \in \Lang(\uu) \}.
\end{aligned}
\]
If $w \notin \Lang(\uu)$, then $\Lext(w) = \Rext(w) = \Bext(w) = \emptyset$.

A factor $w$ is \emph{left special} if $\#\Lext(w) \geq 2$, \emph{right special}
if $\#\Rext(w) \geq 2$, and \emph{bispecial} if it is both left and right special.

Since every factor of $\uu$ has at least one right extension, we have 
\[
  \Delta C(n) = \sum_{w \in \Lang_n(\uu)} \bigl( \#\Rext(w) - 1 \bigr)\quad \text{for every $n \in \mathbb{N}$}.
\]
If $\uu$ is recurrent, then every factor has at least one left extension, and analogously,
 \[
  \Delta C(n) = \sum_{w \in \Lang_n(\uu)} \bigl( \#\Lext(w) - 1 \bigr)\quad \text{for every $n \in \mathbb{N}$}.
 \]

By the classical theorem of Morse and Hedlund~\cite{MorseHedlund1938}, a sequence $\uu$ is eventually
periodic if and only if $C(n+1) = C(n)$ for some $n \in \mathbb{N}$. Consequently,
every aperiodic sequence satisfies $C(n) \geq n + 1$ for all $n \in \mathbb{N}$.
Aperiodic sequences with $C(n) = n + 1$ for all $n \in \mathbb{N}$ are called
\emph{Sturmian sequences}; they are necessarily binary and uniformly recurrent.
A sequence $\uu$ is called \emph{quasi-Sturmian} if there exist constants $\kappa , \ n_0 \in \mathbb{N}$ such that $C(n) = n + \kappa$ for all $n \geq n_0$.

\subsection{Morphisms and antimorphisms}

A mapping $\varphi : \mathcal{A}^* \to \mathcal{B}^*$ is a \emph{morphism} if
$\varphi(vw) = \varphi(v)\varphi(w)$ for all $v, w \in \mathcal{A}^*$, and it is an
\emph{antimorphism} if $\varphi(vw) = \varphi(w)\varphi(v)$ for all $v, w \in \mathcal{A}^*$.
Both are uniquely determined by the images of letters. The action of a morphism $\varphi$
extends naturally to sequences by
$\varphi(u_0 u_1 u_2 \cdots) = \varphi(u_0)\varphi(u_1)\varphi(u_2) \cdots$.
If $\varphi(\uu) = \uu$, then $\uu$ is a \emph{fixed point} of $\varphi$.

The composition of two antimorphisms is a morphism, and the composition of a morphism and
an antimorphism is an antimorphism. Therefore, a group $G$ generated by antimorphisms
over $\mathcal{A}$ consists of morphisms and antimorphisms; as explained in
Remark~\ref{rem:permutations} below, each of them acts on $\mathcal{A}$ as a permutation of letters.
The most prominent antimorphism is the \emph{reversal} (or \emph{mirror image})
$R : \mathcal{A}^* \to \mathcal{A}^*$,
\[
  R(w_1 w_2 \cdots w_n) = w_n \cdots w_2 w_1 .
\]
Over the binary alphabet $\{0,1\}$, we write $\overline{0} = 1$ and $\overline{1} = 0$ and
denote by $\overline{w} = \overline{w_1}\,\overline{w_2} \cdots \overline{w_n}$ the
\emph{complement} (letter exchange) of $w = w_1 w_2 \cdots w_n$. The \emph{exchange antimorphism}
$E : \{0,1\}^* \to \{0,1\}^*$ is the antimorphism determined by $E(0) = 1$ and $E(1) = 0$, i.e.,
\[
  E(w_1 w_2 \cdots w_n) = \overline{w_n} \cdots \overline{w_2}\,\overline{w_1} .
\]
Clearly, $R^2 = E^2 = I$, where $I$ is the identity, and the composition
$ER = RE$ is the morphism exchanging letters, i.e., $ER(w) = \overline{w}$.
Hence $E(w) = R(\overline{w}) = \overline{R(w)}$ for every $w \in \{0,1\}^*$.

Let $\vartheta$ be an antimorphism. A word $w$ is a \emph{$\vartheta$-palindrome} if
$w = \vartheta(w)$. The $R$-palindromes are the classical palindromes; the $E$-palindromes
are sometimes called \emph{antipalindromes}. Note that every non-empty $E$-palindrome has even length.

Let $G$ be a group generated by antimorphisms over $\mathcal{A}$. We say that the language
of a sequence $\uu$ is \emph{closed under $G$} if $g(w) \in \Lang(\uu)$ for every
$w \in \Lang(\uu)$ and every $g \in G$. If $G = \langle \vartheta \rangle$, we simply say that
the language is closed under $\vartheta$.
If the language of a sequence $\uu$ is closed under $G$, then $\uu$ is known to be recurrent.
It suffices to explain that every prefix $w$ of $\uu$ occurs twice. Consider an antimorphism $\vartheta \in G$. By closedness under $G$, $\vartheta(w)$ is a factor of $\uu$, then $\vartheta(w)a$ is a factor of $\uu$ for some letter $a$ and again by closedness under $G$
and since $G$ is a group, $\vartheta^{-1}(\vartheta(w)a)=\vartheta^{-1}(a)w$ is a factor of $\uu$. Hence, $w$ has also a~non-prefix occurrence in $\uu$.

\section{Generalized reflection complexity}\label{sec:generalized}

Let $G$ be a  group generated by antimorphisms over an alphabet $\mathcal{A}$.  Two words $u,v$ over $\mathcal A$ are $G$-{\textit{equivalent}} if $u=g(v)$ for some $g \in G$. We write $u\sim_G v$.

\begin{rem}\label{rem:permutations} In general, an antimorphism is a mapping $\psi:\mathcal A^*\mapsto \mathcal A^*$ satisfying $\psi(uv) = \psi(v)\psi(u)$ for any $u,v \in \mathcal A^*$, thus the images of letters under $\psi$ may be arbitrary words. But if the antimorphisms generate a group $G$, then each antimorphism must have an inverse in $G$; consequently, the letter images must be again letters. In other words, each such antimorphism corresponds to a~letter permutation. Clearly, any group $G$ generated by antimorphisms is finite.
\end{rem}

\begin{defi} Let $G$ be a group generated by antimorphisms over an alphabet $\mathcal{A}$ and let $\uu$ be a sequence over $\mathcal A$. The {\em{generalized reflection complexity with respect to $G$}} (or {\em{$G$-complexity}}, for short) is a mapping ${\mathcal C}_G: \mathbb N \to \mathbb  N$, where ${{\mathcal C}_G}(n)$ is the number of distinct factors of $\uu$ of length $n$, up to equivalence by $\sim_G$.
\end{defi}

Obviously, 
\begin{equation}\label{eq:obviousIneq}\tfrac{1}{\# G}\,{{\mathcal C}}(n)\ \leq \  {{\mathcal C}_G}(n)\ \leq \  {{\mathcal C}}(n)\quad \text{for every $n \in \N$}.\end{equation}

Denote by $ \langle\uu\rangle_n $  the set of equivalence classes of words in $\mathcal{L}_n(\uu)$, i.e., $t \in \langle\uu\rangle_n $
means that $ t = \{g(w): g \in G\}\cap \mathcal{L}_n(\uu)$ for some $w \in \mathcal{L}_n(\uu)$.

If $G$ is generated only by the reflection $R$, then  the  generalized reflection complexity coincides with the reflection complexity introduced by Allouche et al.~\cite{AlCaLiShSt2025}.

\bigskip

It follows directly from the definition that the factor complexity of any sequence is a non-decreasing function, i.e., ${\mathcal C}(n+1)\geq {\mathcal C}(n)$ for all $n\in \mathbb N$. 
This no longer holds for the reflection complexity, where $r(n+1)+1\geq r(n)$ for all $n\in \mathbb N$ and there are sequences such that $r(n)>r(n+1)$ for infinitely many lengths $n$~\cite{AlCaLiShSt2025}.
The generalized reflection complexity behaves like the reflection complexity.
\begin{proposition}  Let $G$ be a group generated by antimorphisms over an alphabet $\mathcal{A}$ and let $\uu$ be a sequence over $\mathcal A$. Then the generalized reflection complexity of $\uu$ satisfies
$$C_G(n+1)+1\geq C_G(n)\quad \text{for all $n\in\mathbb N$.}$$
\end{proposition}
\begin{proof}
For each length $n$, we construct a~non-oriented graph with the set of vertices equal to $ \langle\uu\rangle_n $ and the set of edges  $ \langle\uu\rangle_{n+1}$. A~class $z\in  \langle\uu\rangle_{n+1}$ forms an edge between the classes $s, t\in  \langle\uu\rangle_{n}$ if there exist $u\in t$, $v\in s$ and $w \in z$ such that $u$ is a prefix of $w$ and $v$ is a suffix of $w$ or vice-versa. It follows by the construction that the graph is connected.
Consequently, the number of edges is greater than or equal to the number of vertices reduced by one, in other words $C_G(n+1)\geq C_G(n)-1$.
\end{proof}

\begin{example} Let us consider the binary alphabet $\{0,1\}$ and the group $H$ generated by two antimorphisms: the reflection $R$ and the antimorphism 
$E$ induced by $E(0) = 1$ and $E(1)=0$. We will define a~sequence $\uu$ such that ${\mathcal C}_H(n+1)<{\mathcal C}_H(n)$ for some $n\in \mathbb N$.

Consider the sequence $\uu = \pi({\bf f})$, \ where ${\bf f}$ is the Fibonacci sequence, i.e., the fixed point of $\varphi: 0\mapsto 01,\ 1\mapsto 0$, and 

$$ \pi: \left\{\begin{array}{ll}
    0\mapsto 01;\\
    1\mapsto 10.
    \end{array}\right. $$
    The reader may easily verify that ${\mathcal C}_H(4)=4$ since the factors of $\uu$ of length four are: $$0110\sim_H 1001,\ 1100\sim_H 0011,\ 0010\sim_H 1011, \ 0101.$$
    While ${\mathcal C}_H(5)=3$ as the factors of $\uu$ of length five are: $$01100\sim_H 10011\sim_H 11001\sim_H 00110, \ 10010\sim_H 10110, \ 00101\sim_H 01011.$$  
    In Example~\ref{ex:TMimageFib}, we even show that for this sequence $\uu$ we have ${\mathcal C}_H(n+1)<{\mathcal C}_H(n)$ for every even $n$.
\end{example}
\begin{defi}  Let $G$ be a group generated by antimorphisms over an alphabet $\mathcal{A}$, let $\uu$ be a sequence over $\mathcal A$ and $n \in \mathbb{N}$. Let $t \in  \langle\uu\rangle_n $ and $w \in t$. By $\mathcal{T}(t)$ we denote the number of equivalence classes of the set 
$${\rm Bext}_G(w)=\bigcup_{g \in G}{\rm Bext}(g(w))\,.$$

\end{defi}
Knowing the values $\mathcal{T}(t)$  allows us to calculate the generalized reflection complexity of $\uu$ because 

\begin{equation}\label{eq:vypocetr(n+2)}
   {\mathcal C}_G(n) = \sum_{t \in  \langle\uu\rangle_n } 1 \qquad \text{and}\qquad {\mathcal C}_G(n+2) = \sum_{t \in  \langle\uu\rangle_n } \mathcal{T}(t). 
\end{equation}

\begin{rem}\label{rem:simpleObservations} If  $w$ is a prefix of $\uu$ and $g(w)$ does not have an occurrence  $j>0$ for any $g\in G$, then $\mathcal{T}(t)=0$. Otherwise, ${\rm Bext}_G(w) $ is not empty and hence $\mathcal{T}(t)\geq 1$.    
\end{rem} 

\begin{defi}\label{def:IndexRE} 
Let $G$ be a group generated by antimorphisms over an alphabet $\mathcal A$ and let $\uu$ be a sequence over $\mathcal A$. For each $n \in \mathbb{N}$, we denote by ${\rm Ind}(n)$ the minimum $i \in \N$ having the property: 
  there exists $j >i$ such that the factors of length $n+1$ occurring in $\uu$ at positions $i$ and $j$ are $G$-equivalent. 
\end{defi}

For each $n\in \mathbb N$, ${\rm Ind}(n)$ is well defined. 
If $\uu$ is a recurrent sequence, then ${\rm Ind}(n)=0$ for each $n\in \mathbb N$. If $\uu$ is not recurrent, then, since there are only finitely many factors of length $n+1$, there exists an index $i\in \mathbb N$ such that the factor of length $n+1$ occurring at $i$ has another occurrence $j >i$ in $\uu$. 
Obviously, if  ${\rm Ind}(n-1)>0$, then ${\rm Ind}(n)>0$, too. 

\begin{example} Let us consider the group $H=\langle E, R \rangle$ over the binary alphabet.
We will examine two non-recurrent sequences.

\noindent For the sequence $001010101010\cdots=0(01)^\omega$, we have ${\rm Ind}(n)=1$ for each $n\geq 1$.\\
\noindent For the sequence $11010010100100101\cdots=11\mathbf f$, where $\mathbf f$ is the Fibonacci sequence, we have ${\rm Ind}(n)=1$ for each $n\geq 5$ since the sequence $1\mathbf f$ is known to be recurrent and $g(110100)$ for any $g\in H$ is not its factor. It is readily seen that ${\rm Ind}(n)=0$ for $n\in\{0,1,2,3,4\}$ since for every prefix $w$ of length $n\leq 5$, the word $\overline w$ is a factor of $1\mathbf f$. 
\end{example}

\section{Generalized reflection complexity and periodicity}\label{sec:periodicity}

The aim of this section is to characterize eventually periodic sequences by generalized reflection complexity.

Let us recall a similar characterization of eventually periodic sequences by factor complexity. 

\begin{thm}[\cite{MorseHedlund1938}]\label{thm:MorseHedlund} A~sequence $\uu$ is eventually periodic if and only if there exists $n \in \N$ such that ${\mathcal C}(n+1) ={\mathcal C}(n)$.  
\end{thm}

Let us also recall a result from~\cite{DvPe2026} on the characterization of eventually periodic sequences by reflection complexity. 
\begin{thm}[\cite{DvPe2026}, Theorem 11]\label{thm:r_periodicity}  A~sequence $\uu$ is eventually periodic if and only if there exists $n \in \N$ such that $r(n+2) = r(n)$.  
\end{thm}

Let us state an auxiliary lemma.
\begin{lemma}\label{lem:IndexNEEQ} Let $G$ be a group generated by antimorphisms over an alphabet $\mathcal A$ and let $\uu$ be a sequence over $\mathcal A$. Let $v \in \mathcal{A^*}$,  $i={\rm Ind}(|v|)>0$   and $g \in G$. Assume that for some letters $a,b,x,d \in \mathcal{A}$,   the factor  $xvd$ occurs in $\uu$ at the position $i-1$   and the factor $g(avb)$ occurs in $\uu$ at a position $\geq i$. Then $xvd \nsim_G avb$. 
In particular, $\mathcal T(t_{\rm Ind})\geq 2$, where $t_{\rm Ind}$ is the class containing $v$.
    \end{lemma}
\begin{proof} Note that $g(av)$ and $g(vb)$ occur at a position $\geq i$. Definition of $i={\rm Ind}(|v|)$  ensures that no factor $G$-equivalent to $xv$ occurs at the position $\geq i$.  

Assume, contrary to what we want  to prove, that $xvd \sim_G avb$. It means that $xvd = h(avb)$ for some $h \in G$. 
We discuss two cases:
\begin{itemize}
    \item 
If  $h$ is a morphism, then $xv = h(av)$. Since $h(av)\sim_G g(av)$, we  
deduce that a factor $g(av) \sim_G  xv$ occurs at a position $\geq i$, a~contradiction.

\item If  $h$ is an antimorphism, then $xv = h(vb)$. Since $h(vb)\sim_G g(vb)$, we  
deduce that a factor $g(vb) \sim_G  xv$ occurs at a position $\geq i$, a~contradiction, too.
\end{itemize}
In order to prove that $\mathcal T(t_{\rm Ind})\geq 2$, where $t_{\rm Ind}$ is the class containing $v$, it remains to show that there exist $a,b\in \mathcal A$ and $g\in G$ such that the word $g(avb)$ occurs at a position $\geq i$. 
By definition of $i$, there exists $h\in G$ such that $h(vd)$ occurs at a position $>i$. Taking the left extension of $h(vd)$ if $h$ is a morphism, or the right extension if $h$ is an antimorphism, produces a~word of the form $h(xvd)$ for some $x\in\mathcal A$ occurring at a position $\geq i$, i.e., we may set $g=h$, $a=x$, $b=d$. 

\end{proof}

\begin{proposition}\label{prop:Neklesa}  For every sequence $\uu$ over $\mathcal{A}$ and every group $G$ generated by antimorphisms over $\mathcal{A}$,
$$ \mathcal{C}_G(n+2)\geq \mathcal{C}_G(n)
\quad \text{for every } n \in \N .$$    
\end{proposition}
\begin{proof} 

Consider an arbitrary $n\in\mathbb N$. If ${\rm Ind}(n)=0$, then by Remark~\ref{rem:simpleObservations}, $\mathcal{T}(t)\geq 1$ for each equivalence class $t\in \langle \uu\rangle_n$. Then the formula \eqref{eq:vypocetr(n+2)} implies the statement.
If $i={\rm Ind}(n)>0$, denote by $v$ the factor of $\uu$ of length $n$ occurring at the position $i$. By Lemma~\ref{lem:IndexNEEQ}, $\mathcal T(t_{\rm Ind})\geq 2$, where $t_{\rm Ind}$ is the class containing $v$. 
By Remark \ref{rem:simpleObservations}, $\mathcal{T}(t)=0$ for at most one G-equivalence class and $\mathcal{T}(t)\geq 1$ for all other classes in $\langle \uu\rangle_n$. Consequently, the formula \eqref{eq:vypocetr(n+2)} again implies the statement.

\end{proof}
\begin{corollary}\label{cor:periodicity} Let $G$ be a group generated by antimorphisms over $\mathcal{A}$. A~sequence $\uu$ over $\mathcal{A}$ is eventually periodic if and only if both sequences $\bigl( \mathcal{C}_G(2n)\bigr)_{n\in \N}$  and $\bigl( \mathcal{C}_G(2n+1)\bigr)_{n\in \N}$ are eventually constant.

In other words, $\uu$ is eventually periodic if and only if 
 $\mathcal{C}_G(n+2) = \mathcal{C}_G(n)$ for every sufficiently large $n \in \N$.   
 \end{corollary}
\begin{proof} The previous proposition says that the two  mentioned sequences are non-decreasing.   By \eqref{eq:obviousIneq}, the sequences are bounded if and only if the factor complexity $\mathcal{C}$ of $\uu$ is bounded as well. By Theorem~\ref{thm:MorseHedlund}, $\mathcal C$ is bounded if and only if $\mathcal C$ is eventually periodic.  Each  bounded integer non-decreasing sequence is eventually constant.   
    \end{proof}    

In the sequel, we want to find an analogue of Theorem~\ref{thm:r_periodicity}, i.e., we want to determine assumptions such that the equality $\mathcal{C}_G(n+2) = \mathcal{C}_G(n)$ for one single $n$ guarantees periodicity.

We will make use of a ``folklore'' statement in combinatorics on words, proven in~\cite{DvPe2026}. 

 \begin{proposition}[\cite{DvPe2026}, Proposition 10]\label{pro:RSaLS} Let $\uu$ be an aperiodic sequence. 
 
 \begin{enumerate}
 
 \item
 There exists a right-sided sequence ${\bf x}_L = x_0x_1x_2 \cdots$  and two distinct letters $a,b \in \mathcal{A}$ such that every prefix $w$  of ${\bf x}_L$ is a left special factor of $\uu$ and both $aw$ and $bw$ occur infinitely many times in $\uu$.     
 
 \item 
 There exists a left-sided sequence ${\bf x}_R = \ldots y_2y_1y_0 $  and two distinct letters $c,d \in \mathcal{A}$ such that every suffix $u$  of   ${\bf x}_R$ is a right special factor of $\uu$ and both $uc$ and $ud$ occur infinitely many times in $\uu$.   
\end{enumerate}
 
 \end{proposition}

\begin{defi} We say that a group $G$  generated by antimorphisms over $\mathcal{A}$ has  Property $\textswab{P}$ if  one of the following assumptions is met: 
\begin{enumerate}
\item $\#\mathcal{A} = 2$; 
\item  for every $g \in G$, the restriction of $g$ to $\mathcal{A}$ is either the identity, or it is a permutation without a fixed point and all its cycles are of odd length.
\end{enumerate}
    \end{defi}

\begin{lemma}\label{lem:dvetridyPropP} Let the group $G$ generated by antimorphisms over $\mathcal{A}$ have Property $\textswab{P}$ and $a,b \in \mathcal{A}$, $a\neq b$. Then    $awc\nsim_G bwc$  for every non-empty word $w \in \mathcal{A}^*$ and every $c\in\mathcal A$.  
\end{lemma}
\begin{proof}
Assume the opposite situation, i.e., that there exists $h \in G$ such that 
$awc = h(bwc)$. 

\begin{itemize} \item If $h$ is a morphism, then $awc= h(b)h(w)h(c)$. In particular,   \begin{equation}\label{eq:vylouceniLS}h(c)= c\qquad \text{and} \qquad a=h(b)\neq b.\end{equation}
We can see that $h$ restricted to $\mathcal{A}$ has a fixed point $c$, but $h$ is not the identity.  Hence the second point of our assumptions is not met. A~permutation with this property is excluded on binary alphabet as well. 
 \item If $h$ is an antimorphism, then $awc= h(c)h(w)h(b)$. In particular, 
\begin{equation}\label{eq:vylouceniLSanti}h(c)= a\neq b\qquad \text{and} \qquad c=h(b).\end{equation}
This is not possible on binary alphabet. Hence $\#\mathcal{A}\geq 3$ and $h$ is not the identity, as $b\neq a$.   Moreover, $h(w)=w$. Let us write $w = w_1w_2\cdots w_n$.   If $n=1$, then  $h(w)=w$ implies that $w_1$ is a fixed point of the permutation, which is excluded by    $\textswab{P}$. If $n\geq 2$, then  $h(w_1) = w_n$ and $h(w_n) = w_1$, in other words, if $w_1\not =w_n$, then the permutation has a cycle of length $2$; if $w_1=w_n$, then the permutation has a fixed point $w_1$ - both are excluded by  $\textswab{P}$.

\end{itemize}

The assumption  $awc = h(bwc)$ for some $h \in G$ led to a contradiction, so we can conclude that $awc\nsim_G bwc$. 
\end{proof}

\begin{proposition}\label{pro:neperiodicity} Let the group $G$ generated by antimorphisms over $\mathcal{A}$ have Property $\textswab{P}$. Then 
the generalized reflection complexity of an aperiodic sequence  $\uu$ over $\mathcal{A}$ satisfies  
\begin{equation}\label{eq:prirustek}\mathcal{C}_G(n+2)\geq \mathcal{C}_G(n)+1 \ \ \ \text{for every $n \in \N$, $n \geq 1$.}\end{equation}   
\end{proposition}
\begin{proof} 
Let $wc$ be a prefix of length $n+1$ of the sequence ${\bf x}_L$; by Proposition \ref{pro:RSaLS}, $awc$ and $bwc$ occur infinitely many times in $\uu$.

We will denote by $t_{\rm Ind}$ the $G$-equivalence class containing the factor $v$ of length $n$ occurring at the position  ${\rm Ind}(n)$. By $t_L$ we denote the $G$-equivalence class containing the prefix $w$ of ${\bf x}_L $ having the length $n$.  

 \begin{enumerate}
     \item  ${\rm Ind}(n)=0$ : then $\mathcal{T}(t)\geq 1$ for every $G$-equivalence class $t$ and, by Lemma~\ref{lem:dvetridyPropP}, the class $t_L$ satisfies $\mathcal{T}(t_L)\geq 2$. The formula \eqref{eq:vypocetr(n+2)} implies  $\mathcal{C}_G(n+2)\geq  \mathcal{C}_G(n)+1$. 

     \item  ${\rm Ind}(n) >0$  and  $v\nsim_G w$ : then  $t_L \neq t_{\rm Ind}$. 
Recall that the inequality $\mathcal{T}(t)\geq 1$  applies to every $G$-equivalence class with one exception.  
      By Lemmata~\ref{lem:dvetridyPropP} and  \ref{lem:IndexNEEQ},    $\mathcal{T}(t_L)\geq 2$ and  $\mathcal{T}(t_{\rm Ind})\geq 2$.  The formula \eqref{eq:vypocetr(n+2)} gives  $\mathcal{C}_G(n+2)\geq  \mathcal{C}_G(n)+1$. 
   \item  ${\rm Ind}(n) >0$  and  $v\sim_G w$ : then  $t_L = t_{\rm Ind}$.   
 Denote by $xvd$  the factor of length $n+2$ occurring at the position  ${\rm Ind}(n)-1$. By Proposition \ref{pro:RSaLS},  $awc$ and $bwc$ occur infinitely many times in $\uu$.  Since $v\sim_G w$, we have $w=h(v)$ for some $h\in G$ and $awc=h(a'vb')$ for some $a',b'\in \mathcal A$ and similarly for $bwc$. Consequently, Lemma \ref{lem:IndexNEEQ} says that $xvd
 \nsim_G awc$ and $xvd
 \nsim_G bwc$. Since $awc\nsim_G bwc$, the set ${\rm Bext}_G(v)$ contains at least $3$ elements which are mutually not $G$-equivalent. Hence $\mathcal{T}(t_L)\geq 3$, which again forces the inequality   $\mathcal{C}_G(n+2)\geq  \mathcal{C}_G(n)+1$.   
 \end{enumerate}
\end{proof}

\begin{rem}\label{rem:0} The inequality \eqref{eq:prirustek}  is deduced for $n \geq 1. $  Let us explain why the case $n=0$ is not considered. 

\begin{description}
    \item[Ternary case] Consider the ternary sequence $\uu = \varphi({\bf f}) $, where ${\bf f}$ is the Fibonacci sequence and $\varphi$ is the morphism given by $\varphi(0)=ab, \ \varphi(1) = c$, i.e., 
    $$\begin{array}{rcl}
    \mathbf f&=&010010100100101\cdots \\
    \uu&=&abcababcabcababcababcabc\cdots
    \end{array}$$
 Let $G$  be the group generated  by $R$ and the antimorphism $\theta$ which restricted to the ternary alphabet is the cyclic permutation $a \mapsto b\mapsto c \mapsto  a$. Obviously $\uu$ is aperiodic and $ab, ba, bc, ca $  is the list of  factors of $\uu$ of length $2$.  All these factors are mutually $G$-equivalent. Consequently, $\mathcal{C}_G(2)=1 = \mathcal{C}_G(0)$.   

 \item[Binary case] Here, the inequality \eqref{eq:prirustek} is fulfilled also for $n=0$ since any binary aperiodic sequence contains the factors $01, 10$ and at least one of the words $00$ or $11$. Hence, $\mathcal{C}_G(2)\geq 2$. 
\end{description}
  
\end{rem}

\section{Groups over binary alphabet}\label{sec:binary}
In the rest of the article, we will restrict our attention to the binary alphabet $\{0,1\}$.
There are three possible groups generated by antimorphisms in this case:

\begin{enumerate}
\item $\langle R\rangle$: The generalized reflection complexity coincides with the reflection complexity; the papers~\cite{AlCaLiShSt2025, DvPe2026} study this case.
\item  $\langle E\rangle$: Combining Corollary \ref{cor:periodicity}, Proposition \ref{pro:neperiodicity}, and the binary case of Remark~\ref{rem:0}, we can characterize eventually periodic binary sequences using the generalized reflection complexity for the group $G=\langle E\rangle$. 
\begin{coro} Let $G=\langle E\rangle$. A~binary sequence $\uu$ is eventually periodic if and only if there exists $n \in \N$ such that ${\mathcal C}_G(n+2) = {\mathcal C}_G(n)$. 
\end{coro}

\begin{rem}\label{rem:nikdyekvality} Let us explain that if  $G=\langle E \rangle$, there is no binary sequence satisfying ${\mathcal C}_G(n+2) = {\mathcal C}_G(n)+1$ for all $n\in \mathbb N$. Clearly, ${\mathcal C}_G(1)=1$ since $0\sim_G 1$. 
All possible factors of length three are: $000\sim_G 111, \ 001\sim_G 011, \ 010 \sim_G 101, \ 100 \sim_G 110$.
The reader is invited to check that it is impossible to construct an aperiodic sequence whose factors of length three belong to at most two $G$-equivalence classes.
\end{rem}

\item $H=\langle E, R\rangle=\{I, E, R, ER\}$: Throughout the remaining part of the paper, we will study the $H$-complexity.

As in the second case, combining Corollary \ref{cor:periodicity}, Proposition \ref{pro:neperiodicity}, and the binary case of Remark~\ref{rem:0}, we can characterize eventually periodic binary sequences using the  $H$-complexity. 

\begin{thm}\label{thm:Conjectur27RE} A~binary sequence $\uu$ is eventually periodic if and only if there exists $n \in \N$ such that ${\mathcal C}_H(n+2) = {\mathcal C}_H(n)$. 
\end{thm}
In contrast to Remark \ref{rem:nikdyekvality}, if  $G=\langle R \rangle$ or $G=H$, there are  binary sequences for which ${\mathcal C}_G(n+2) = {\mathcal C}_G(n)+1$ for all $n\in\mathbb N$. 

In the case $G=\langle R \rangle$,  examples of such sequences can be found 
in~\cite{AlCaLiShSt2025} and \cite{DvPe2026}. For  $G=H$,  examples  are provided in the sequel. 

\end{enumerate}

\section{Sequences with language closed under $H$}\label{sec:Hclosed}

Recall that the $R$-palindromic complexity $\mathcal{P}_R: \N\mapsto \N$, resp. the $E$-palindromic complexity $\mathcal{P}_E: \N\mapsto \N$ of a~sequence $\uu$ counts the $R$-palindromes, resp. $E$-palindromes of length $n$ occurring in $\uu$. Formally,

$$\mathcal{P}_R(n) = \#\{w \in \mathcal{L}_n(\uu)\ :\ w = R(w)\};$$

$$\mathcal{P}_E(n) = \#\{w \in \mathcal{L}_n(\uu)\ :\ w = E(w) \}.$$

These values allow us to derive the relationship between the factor complexity and the $H$-complexity in the case when the language is closed under $H$. 
Every class (besides the class containing the empty word) with respect to $\sim_H$  has either 4 or 2 elements: 
On the one hand, if $w$ is neither an $R$-palindrome, nor an $E$-palindrome, then $w, \overline{w}, E(w), R(w)$ are distinct elements of the class containing $w$.
On the other hand, $w$ cannot be both an $R$-palindrome and an $E$-palindrome. If $w$ is an $R$-palindrome, then $R(w)=w\not =\overline{w}=E(w)$, consequently the class  has 2 elements. Similarly if $w=E(w)$. The following lemma summarizes these considerations.

\begin{lemma}\label{lem:vztahCGC} Let $\uu$ be a~binary sequence having the language closed under $H$. Then for every $n\in \N, n\geq 1$,
\begin{equation}\label{eq:CPRPECG}
    \mathcal{C}_H(n) = \frac{1}{4}\, \Bigl(  \mathcal{C}(n)+ \mathcal{P}_R(n)+ \mathcal{P}_E(n)\Bigr)\,.
\end{equation}
\end{lemma}

In the case of languages closed under $H$, there is a~close relation between the $R$-palindromic complexity, the $E$-palindromic complexity and the factor complexity.

\begin{proposition}[\cite{PeSt2013}]\label{prop:Grich} Let $\uu$ be a~binary sequence with language closed under $H$. Then for every $n \in \N$,
\begin{equation}\label{eq:GrichnessRE}\mathcal{P}_R(n)+ \mathcal{P}_R(n+1) + \mathcal{P}_E(n)+ \mathcal{P}_E(n+1)\ \leq \   4+\Delta \mathcal{C}(n),
\end{equation}      
where $\Delta \mathcal{C}(n) = \mathcal{C}(n+1) - \mathcal{C}(n)$ denotes the first difference of $\mathcal{C}$.

\end{proposition}

\begin{defi}\label{def:Grich} Let the language of a~binary sequence $\uu$ be closed under the group $H$. If the equality in~\eqref{eq:GrichnessRE} holds for all $n\in \mathbb N, \ n\geq 1$, $\uu$ is called $H$-\em{rich}.
\end{defi}
For $H$-rich sequences, we deduce the following simpler formula for the $H$-complexity.
\begin{proposition}\label{pro:CGforGrich}
Let $\uu$ be a~binary sequence with language closed under $H$. If $\uu$ is $H$-rich, then for every $n \in \N, n \geq 1$, we have

$$\mathcal{C}_H(n+2) = \mathcal{C}_H(n) +  \tfrac12 \Delta \mathcal{C}(n+1). 
$$

\end{proposition}
\begin{proof} 
By Lemma \ref{lem:vztahCGC}, for $n\geq 1$,
$$ \mathcal{C}_H(n+2) - \mathcal{C}_H(n) = 
 \tfrac14 \,\Bigl( \mathcal{C}(n+2) - \mathcal{C}(n) + \underbrace{\mathcal{P}_R(n+2)+ \mathcal{P}_E(n+2) - \mathcal{P}_R(n)- \mathcal{P}_E(n)}_{:=X}\Bigr).
$$
Note that  in  \eqref{eq:GrichnessRE} the equality holds for every $n\geq 1$. If we replace in the equality  the variable  $n$  by $n+1$ and subtract the original equality,  we deduce 
$$\begin{array}{rcl}
X&=&\mathcal{P}_R(n+2)- \mathcal{P}_R(n) + \mathcal{P}_E(n+2)- \mathcal{P}_E(n)\\
&=& \Delta \mathcal{C}(n+1) -  \Delta \mathcal{C}(n)\\
&=& \mathcal{C}(n+2)-2 \mathcal{C}(n+1)+ \mathcal{C}(n).
\end{array}
$$
Consequently,  $$\mathcal{C}_H(n+2) - \mathcal{C}_H(n) = 
 \tfrac14 \,\Bigl( \mathcal{C}(n+2) - \mathcal{C}(n)+ \mathcal{C}(n+2)-2 \mathcal{C}(n+1)+ \mathcal{C}(n)\Bigr)=\tfrac{1}{2}\Delta {\mathcal C}(n+1),$$
as desired. 

\end{proof}

\subsection{$H$-complexity of complementary symmetric Rote sequences}

Let us study the $H$-complexity of a particular class of binary sequences -- complementary symmetric Rote sequences. We will observe in the sequel that the minimum $H$-complexity for aperiodic recurrent sequences is attained exactly for this class.

\begin{defi}
A \emph{Rote sequence} is a sequence $\mathbf v$ whose factor complexity satisfies
\[
\mathcal C_{\mathbf v}(n)=2n
\qquad\text{for all } n\in\mathbb N,\ n\ge1.
\]
If the language of a Rote sequence $\mathbf v$ is closed under the morphism $ER$ exchanging letters $0\leftrightarrow 1$,
it is called \emph{complementary symmetric} (CS).
\end{defi}
Rote \cite{Rote1994} proved that these sequences are closely related to
Sturmian sequences. In order to recall his characterization, let us denote by $S$ the mapping
\[
S : \{0,1\}^+ \to \{0,1\}^*
\]
such that for every $v_0 \in \{0,1\}$ we put $S(v_0)=\varepsilon$, and for every
\[
v=v_0v_1\cdots v_n \in \{0,1\}^+
\]
of length at least $2$ we define
\[
S(v_0v_1\cdots v_n)=u_0u_1\cdots u_{n-1},
\]
where
\[
u_i = v_i + v_{i+1} \pmod{2}
\quad\text{for all } i\in\{0,1,\ldots,n-1\}.
\]

Moreover, for every $\mathbf v\in\{0,1\}^{\mathbb N}$ we set $S(\mathbf v)=\mathbf u$, where
\[
u_i = v_i + v_{i+1} \pmod{2}
\quad\text{for all } i\in\mathbb N.
\]

\begin{example}
We have
\[
\overline{001110}=110001
\quad\text{and}\quad
S(001110)=S(110001)=01001.
\]
Clearly, the images of $v$ and $\overline{v}$ under $S$ coincide for each
$v\in\{0,1\}^+$.
\end{example}

\begin{proposition}[{\cite{Rote1994}}]\label{pro:RoteVersusSturm}
Let $\mathbf u$ and $\mathbf v$ be two sequences over $\{0,1\}$ such that
$\mathbf u=S(\mathbf v)$. Then $\mathbf v$ is a complementary symmetric
Rote sequence if and only if $\mathbf u$ is a Sturmian sequence.
\end{proposition}

This characterization allows one to show  that any CS Rote sequence  $\vv$ contains infinitely many $R$-palindromes and infinitely many $E$-palindromes. Moreover, $\mathbf v$ is uniformly recurrent since the associated Sturmian sequence $\uu$ is uniformly recurrent.
Let us explain the reason: It suffices to show that any prefix $v$ of $\vv$ occurs with bounded gaps. Denote by $z$ the prefix of $\vv$ having $\overline{v}$ as suffix and of course $v$ as prefix. It exists since $\vv$ is closed under $ER$. The factor $S(z)$ of $\uu$ occurs with bounded gaps since $\uu$ is uniformly recurrent. To each occurrence of $S(z)$ corresponds an occurrence of $z$ or $\overline{z}$ in $\vv$ and both $z$ and $\overline z$ contains $v$. Consequently, $v$ occurs with bounded gaps in $\vv$.

It implies the following corollary.   
\begin{corollary} The language of any complementary symmetric Rote sequence is closed under $H$.  
\end{corollary}

\begin{thm}\label{thm:Rote} The $H$-complexity  of every  CS Rote sequence equals 
$${\mathcal C}_H(n)=\lfloor \tfrac{n}{2} \rfloor+1 \quad \text{for every $n\in \N$}.$$
\end{thm}
\begin{proof} For $n=0, 1, 2$ the formula clearly holds.  By \cite{PeSt2016}, any CS Rote word is $H$-rich and thus we can  use Proposition \ref{pro:CGforGrich}. As $\mathcal{C}(n) = 2n$ for $n\geq 1$,  we get  $\Delta \mathcal{C}(n) = 2$ for  $n \geq 1$. Hence, 

  \centerline{$ \mathcal{C}_H(n+2) = \mathcal{C}_H(n) + 1
$ for  $n\geq 1$.}

 \noindent The formula for ${\mathcal C}_H$ follows.
\end{proof}
\begin{rem} Compare the result from Theorem~\ref{thm:Rote} with the reflection complexity of any CS Rote sequence, derived in~\cite{AlCaLiShSt2025}, 
$$r(n)=n +1 \ \ \  \text{for  every $n\in \N$.}$$
\end{rem}

To close this section, let us determine the $H$-complexity for two Rote sequences that are not complementary symmetric. 

\begin{example} 

Let $\uu$ be the fixed point of the morphism $\varphi$, 
     $$\varphi: \left\{\begin{array}{ll}
    0\mapsto 001\,;\\
    1\mapsto 111\,.
    \end{array}\right.$$
The sequence $\uu$ is Rote~\cite{Rote1994}, i.e., the factor complexity satisfies ${\mathcal C}(n)=2n$ for all $n\in\mathbb N, n\geq 1$. The sequence $\uu$ is recurrent, but not uniformly recurrent as there are blocks of ones of any size. Its language is closed under $R$, but not under $E$ since no factor of length $\geq 5$ has its $E$-image in ${\mathcal L}(\uu)$. 
Consequently, for $n\geq 5$, the $H$-complexity and the reflection complexity coincide. It is not difficult to see that there are two palindromes of every length $n\geq 1$, therefore ${\mathcal C}_H(n)=n+1$ for all $n\geq 5$.
\end{example}

The  Rote sequence  $\uu$ defined in the following example comes from  \cite{Currie2025}, where  Currie and Rampersad show that the critical exponent of $\uu$  is the smallest possible among all Rote sequences.

\begin{example}
Let $\uu = \pi(\vv)$, where $\vv$ is the fixed point of the morphism $\varphi$,

 $$\varphi: \left\{\begin{array}{ll}
    0\mapsto 0121\\
    1\mapsto 021\\
    2 \mapsto 01 
    \end{array}\right. \quad \text{and} \quad  \pi: \left\{\begin{array}{ll}
    0\mapsto 011\\
    1\mapsto 0\\
    2 \mapsto 01 
    \end{array}\right.  \,. $$
The sequence $\uu$ is Rote~\cite{Currie2025}. The language is uniformly recurrent for $\uu$ is a~morphic image of a~fixed point of primitive morphism. Its language is not closed under $R$, $E$ or $ER$ because no factor of length $\geq 22$ has its $R$-image, $E$-image and $ER$-image in ${\mathcal L}(\uu)$. Therefore, for $n\geq 22$, the $H$-complexity and the factor complexity coincide, i.e., ${\mathcal C}_H(n)=2n$ for all $n\geq 22$.   
\end{example}

\subsection{$H$-complexity of the generalized Thue-Morse sequence}

A~\textit{generalized binary Thue-Morse sequence} ${\mathbf t}_{b,2}$ with $b\in \mathbb N, b\geq 2$, is the fixed point of the morphism 
$$\varphi_{b,2}: 0 \mapsto (01)^{b/2}, \ 1 \mapsto (10)^{b/2}\,,$$
where $(01)^{b/2}$ is the prefix of length $b$ of $(01)^\omega$.
In particular, ${\mathbf t}_{2,2}$ is the famous Thue-Morse sequence.
For $b$ odd, the sequence ${\mathbf t}_{b,2}=(01)^{\omega}$.
For $b$ even, the factor complexity of ${\mathbf t}_{b,2}$ is known~\cite{St2012}, see Table~\ref{tab}.
\begin{table}[htbp]
\centering
\begin{tabular}{c|c|c}
$n$ & $\Delta\mathcal{C}(n)$ & $\mathcal{C}(n)$ \\
\hline
0 & $1$ & 1 \\

\hline
$1 \leq n \leq b$ & $2$ & $2n$ \\
\hline
$b^k + 1 + \ell$ & $4$ & $4(n-1) - 2(b^k - b^{k-1})$ \\
$k \ge 1$, $0 \leq \ell < b^k - b^{k-1}$ & & \\
\hline
$(2b-1)b^{k-1} + 1 + \ell$ & $2$ & $4(n-1) - 2(b^k - b^{k-1} + \ell)$ \\
$k \geq 1$, $0 \leq \ell < b^{k+1} - 2b^k + b^{k-1}$ & & \\

\end{tabular}
\caption{Values of $\Delta\mathcal{C}(n)$ and $\mathcal{C}(n)$ of the generalized Thue-Morse sequence $t_{b,2}$ for the aperiodic case with $b$ even.}\label{tab}
\end{table}

\begin{thm}
The $H$-complexity of the generalized binary Thue-Morse sequence ${\mathbf t}_{b,2}$ with $b\in \mathbb N, b\geq 2$, satisfies:
\begin{itemize}
\item For $b$ odd,
$${\mathcal C}_H(n)=1\ \text{for all $n\in \mathbb N$}\,.$$
\item For $b$ even and $n\geq 2b$, we have
$${\mathcal C}_H(n)=\left\{\begin{array}{ll}
    \frac{1}{4}\, \mathcal{C}(n) & \text{for $n$ odd}\,;\\
    1+\frac{1}{4}\mathcal{C}(n+1) & \text{for $n$ even}\,.
    \end{array}\right. $$
\item For $b$ even and $0\leq 2n<2b$, we have 
    $${\mathcal C}_H(2n)={\mathcal C}_H(2n+1)=\left\{\begin{array}{ll}
    n+1 & \text{for  $2n\leq b$}\,;\\
   2n+1-\frac{b}{2} & \text{for $b < 2n$}\,.\\
       \end{array}\right.$$

Moreover, ${\mathcal C}_H(n+2) - {\mathcal C}_H(n) \in \{1,2\}$ for all $n\in \mathbb N$ and both values are attained infinitely many times.
\end{itemize}
\end{thm}
\begin{proof}
The sequence ${\mathbf t}_{b,2}$ is $H$-rich~\cite{St2012}.
For $b$ odd, the sequence ${\mathbf t}_{b,2}=(01)^{\omega}$, therefore ${\mathcal C}_H(n)=1$ for all $n \in \mathbb N$. 
For $b$ even, as shown in~\cite{St2012}, the longest palindromic factor of odd length has the length equal to $2b-1$.
Moreover, $E$-palindromes cannot have odd length. 
Consequently, since the language of ${\mathbf t}_{b,2}$ is closed under $H$, the formula~\eqref{eq:CPRPECG} holds and the equality in~\eqref{eq:GrichnessRE} is attained. We get the following cases.
\begin{itemize}
\item For $n$ odd, $n\geq 2b+1$, we have $ \mathcal{P}_R(n)=\mathcal{P}_E(n)=0$, therefore $${\mathcal C}_H(n)=\frac{1}{4}\, \mathcal{C}(n)\,.$$
\item For $n$ even, $n\geq 2b$, we have $ \mathcal{P}_R(n+1)=\mathcal{P}_E(n+1)=0$, therefore
$$\begin{array}{rcl}
{\mathcal C}_H(n)&=&\frac{1}{4}\, \Bigl(  \mathcal{C}(n)+ \mathcal{P}_R(n)+ \mathcal{P}_E(n)\Bigr)\\
&=&\frac{1}{4}\, \Bigl( \mathcal{C}(n)+ \mathcal{P}_R(n)+ \mathcal{P}_E(n)+\mathcal{P}_R(n+1)+ \mathcal{P}_E(n+1)\Bigr)\\
&=&\frac{1}{4}\, \Bigl( \mathcal{C}(n)+4+\mathcal{C}(n+1) - \mathcal{C}(n)\Bigr)\\
&=&1+\frac{1}{4}\mathcal{C}(n+1)\,. \end{array}$$
\end{itemize}

Using Table~\ref{tab} together with the formula ${\mathcal C}_H(n+2)={\mathcal C}_H(n)+\frac{1}{2}\Delta{\mathcal C}(n+1)$ for $n\geq 1$, we get the values of ${\mathcal C}_H(n)$ for $n\leq 2b-1$ and the fact that ${\mathcal C}_H(n+2)-{\mathcal C}_H(n)\in \{1,2\}$ (${\mathcal C}_H(2)={\mathcal C}_H(0)+1$ clearly holds).
\end{proof}

\subsection{Asymptotic relationship between ${\mathcal{C}_H(n)}$ and ${\mathcal{C}(n)}$
}
Let us recall a~conjecture on the asymptotic relationship between the reflection and factor complexity stated by Allouche et al.~\cite{AlCaLiShSt2025}.

\begin{conjecture}\label{con:AS} If $\uu$ is an aperiodic sequence such that $L:=\lim\limits_{n \to +\infty}{{r}(n)}/{\mathcal{C}(n)}$ exists, then $L=\frac12$ or $L=1$. 
\end{conjecture}
Under the assumption of closedness of the language under $H$, we prove an analogous statement for the $H$-complexity. 

\begin{thm}\label{thm:limita}  Let the language of an aperiodic sequence  $\uu$ be closed under $H$. If  $\lim\limits_{n \to +\infty}{\mathcal{C}_H(n)}/{\mathcal{C}(n)}$ exists,    then 
$$\lim_{n \to +\infty}\frac{\mathcal{C}_H(n)}{\mathcal{C}(n)} = \frac14.$$
\end{thm}
\begin{proof}

Using 
Lemma \ref{lem:vztahCGC}  and the inequality \eqref{eq:GrichnessRE},
 we can write 
$$
\frac{\mathcal{C}_H(n)}{\mathcal{C}(n)}  =\frac{ \mathcal{C}(n)+ \mathcal{P}_R(n)+  \mathcal{P}_E(n)}{4\,\mathcal{C}(n)}  \ \geq  \ \frac14  $$
 and 
$$\frac{\mathcal{C}_H(n)}{\mathcal{C}(n)} \leq \frac{4+\mathcal{C}(n+1)- \mathcal{P}_R(n+1) -\mathcal{P}_E(n+1)  }{4\,\mathcal{C}(n)} \leq   \frac{1}{\mathcal{C}(n)} + \frac{\mathcal{C}(n+1)}{4\,\mathcal{C}(n)}. 
$$
The sequence $\uu$ is aperiodic and thus  $\mathcal{C}(n)$ tends to infinity.  Hence   \begin{equation}\label{eq:infsup}\frac14\leq \lim\limits_{n\to +\infty}  \frac{\mathcal{C}_H(n)}{\mathcal{C}(n)}   \leq  \liminf \limits_{n\to +\infty}   \frac{\mathcal{C}(n+1)}{4\,\mathcal{C}(n)}\,.\end{equation}

Since $\mathcal{C}(n+m)\leq \mathcal{C}(n)\mathcal{C}(m)$, we can apply Fekete's lemma for submultiplicative sequences which guarantees the existence of 
$$\lim_{n \to +\infty} \sqrt[n]{\mathcal{C}(n)} = \inf \{ \sqrt[n]{\mathcal{C}(n)} : n \in \mathbb{N}\}\,.$$ 

Let us discuss the following two cases. 
\bigskip

\noindent {\bf Case } $\lim\limits_{n \to +\infty} \sqrt[n]{\mathcal{C}(n)}=1$.  \quad By the Stolz-Ces\`aro theorem
\begin{equation}\label{eq:Cesaro}
\liminf_{n \to +\infty} \frac{\mathcal{C}(n+1)}{\mathcal{C}(n)} \leq \liminf_{n \to +\infty} \sqrt[n]{\mathcal{C}(n)}  =\lim _{n \to +\infty} \sqrt[n]{\mathcal{C}(n)}\,.
\end{equation}
 The statement follows by combination of inequalities  \eqref{eq:Cesaro} and \eqref{eq:infsup}. 

\bigskip

\noindent {\bf Case } $\lim\limits_{n \to +\infty} \sqrt[n]{\mathcal{C}(n)}=a>1$. \quad Equivalently, for  every $\delta\in (0,a)$ there exists $n_0 \in \N$ such that for every $n > n_0$ 
\begin{equation}\label{eq:bounds}
  (a-\delta)^n < \mathcal{C}(n) < (a+\delta)^n.   
\end{equation}

Observe that $R$-palindromes and $E$-palindromes are described by their left (or right) half. Thus  $\mathcal{P}_R(n)$ and $\mathcal{P}_E(n)$ are bounded from above by $ \mathcal{C}\bigl( \lfloor \frac{n+1}{2}\rfloor\bigr)$. Hence 
$$
 \frac{\mathcal{C}_H(n)}{\mathcal{C}(n)} =\frac{ \mathcal{C}(n)+\mathcal{P}_R(n)+ \mathcal{P}_E(n)}{4\,\mathcal{C}(n)} \leq \frac{ \mathcal{C}(n)+ 2\mathcal{C}\bigl( \lfloor \frac{n+1}{2}\rfloor\bigr)}{4\,\mathcal{C}(n)} = \frac14 + \frac{ \mathcal{C}\bigl( \lfloor \frac{n+1}{2}\rfloor\bigr)}{2\,\mathcal{C}(n)}\,. 
$$
To complete the proof, it is enough to show that the limit of the last term in the previous inequality is zero.

Using \eqref{eq:bounds},  we can  estimate for every sufficiently small $\delta >0$,
$$
0\leq \frac{ \mathcal{C}\bigl( \lfloor \frac{n+1}{2}\rfloor\bigr)}{\mathcal{C}(n)}  \leq \frac{(a+\delta)^{\frac{n+1}{2}}}{(a-\delta)^n} = (a+\delta)^{\frac12}\, A^\frac{n}2, \text{ where } 
A = \frac{a+\delta}{(a-\delta)^2}\,.
$$
Note that $\lim\limits_{\delta\to 0} \frac{a+\delta}{(a-\delta)^2} = \frac{1}{a} <1$, and thus we can choose $\delta$ small enough such that  $A< 1$ and consequently,  $\lim\limits_{n\to +\infty}\frac{ \mathcal{C}\bigl( \lfloor \frac{n+1}{2}\rfloor\bigr)}{\mathcal{C}(n)} = 0$.  

\end{proof}
\begin{rem}  Shallit~\cite{Shallit2026} studied the relation between the palindromic complexity  and the factor complexity. He showed that 
$$\lim\limits_{n \to +\infty}\frac{\mathcal{P}_R(n)\bigl(\log(\mathcal{P}_R(n)+1\bigr)}{\mathcal{C}(n)} = 0$$ for each aperiodic sequence $\uu$.

Let us recall that for $\uu$ whose language is closed under reversal, we can write 
$$ \frac{r(n)}{\mathcal{C}(n)} = \frac{\mathcal{C}(n)+\mathcal{P}_R(n)}{2\,\mathcal{C}(n)} = \frac12 + \frac{\mathcal{P}_R(n)}{2\,\mathcal{C}(n)}.   
$$
By the aforementioned result of Shallit, $\lim\limits_{n \to +\infty} \frac{r(n)}{\mathcal{C}(n)}=\frac{1}{2}$.
Therefore, Conjecture \ref{con:AS} is valid for aperiodic sequences with languages closed under reversal.

Note that if one succeeded in deriving a~similar result for the $E$-palindromic complexity ${\mathcal P}_E(n)$, then Theorem \ref{thm:limita} would hold without the assumption on the existence of the limit $\lim\limits_{n \to +\infty}\frac{\mathcal{C}_H(n)}{\mathcal{C}(n)}$.

\end{rem}

\section{Aperiodic sequences with the minimum\\ \mbox{$H$-complexity}}\label{sec:minimal}

 Aperiodic sequences with the minimum factor complexity are called Sturmian sequences. They also play an important role in the description of aperiodic sequences with the minimum reflection complexity. 
 \begin{thm}[\cite{DvPe2026}, Theorem 19]\label{thm:Rminimum} Let $\uu$ be a sequence. Then $r(n+2) = r(n)+1$ for every $n \in \N$ if and only if
\begin{itemize}
    \item either $\uu$ is a Sturmian sequence,

\item or there exist a Sturmian sequence $\mathbf v$ over $\{a,b\}$ and a letter $c \not \in \{a,b\}$  such that 
 $\uu = \pi(\mathbf v)$ or $\uu = c\pi(\mathbf v)$, where $\pi$ is a morphism given by 

\medskip 
\centerline{$\pi: a \mapsto ac, \ b\mapsto bc$. }
  
\end{itemize}    
\end{thm}
 Here, we again focus on binary sequences and the $H$-complexity.  We will show that in the description of aperiodic sequences with the  minimum $H$-complexity,  i.e., where ${\mathcal C}_H(n+2)={\mathcal C}_H(n)+1$ for all $n \in \mathbb N$, a~similar role is played by complementary symmetric Rote sequences. The following lemmas will serve this purpose.
\begin{lemma}\label{lem:2veStejneTride}    Let $u,  v$ be two different factors of length $n\geq 1$ belonging to the same class of $H$-equivalence $t$. Assume that  for some letters  $c,d\in \{0,1\}$ the factors $vc$ and $ud$ are left special.  If $\mathcal{T}(t) = 2$, then  $ud = \overline{vc}$.  
\end{lemma}
\begin{proof} 
Clearly,  $0vc, 1vc, 0ud, 1ud \in {\rm Bext}_H(v)$ and by Lemma \ref{lem:dvetridyPropP}, $0vc\nsim_H  1vc$,  $0ud\nsim_H 1ud$.  As ${\rm Bext}_H(v)$ contains only 2 classes,   we have to discuss two scenarios.

\begin{itemize}

  \item $0vc\sim_H 0ud$ and  $1vc\sim_H 1ud$: \quad 
 If both equivalences were realized by the antimorphism $E$, i.e., $0vc = E(0ud) =\overline{d}E(u)1$ and $1vc = E(1ud) =\overline{d}E(u)0$, it would lead to a~contradiction $c=1$ and $c=0$. 
 
 If both equivalences were realized by the antimorphism $R$, i.e., $0vc = R(0ud) =dR(u)0$ and $1vc = R(1ud) =dR(u)1$, it would lead to a~contradiction $c=0$ and $c=1$. 

 Obviously, none of the two equivalences can be realized by $ER$ or by identity. Hence one of the equivalence is realized by $E$  and one by $R$. It gives $u=E(v)$ and $u=R(v)$,  i.e., $v=\overline{v}$, which is a~contradiction, too. 

     \item $1vc\sim_H 0ud$ and  $0vc\sim_H 1ud$: \quad    If one of the equivalences is realized by $ER$, the proof is completed. 
     
     Assume that no equivalence is realized by $ER$. Analogously to the previous case, all other possibilities lead to a contradiction. 
\end{itemize}
\end{proof}

\begin{lemma}\label{lem:simpleObservations1}  Let $wc$ be a left special factor of a binary sequence $\uu$ with  non-empty $w$ and $c \in \{0,1\}$.  Let $t$ denote the $H$-equivalence class containing $w$.  Then $\mathcal{T}(t) \geq 2$. Moreover,   \begin{enumerate} 

\item   \quad 
     if $\mathcal{T}(t)= 2$, then  ${\rm Bext}_H(w) \subset \{g(0wc) : g \in H\} \cup  \{ g(1wc) :  g \in H\}$; 
\item  \quad  if  $\mathcal{T}(t)=2$  and $dw\overline{c}$ occurs in $\uu$ for some letter $d \in \{0,1\}$,  then $w=R(w)$ or $w=E(w)$.   
\end{enumerate}

\end{lemma}
\begin{proof}
Both factors   $0wc$ and $1wc$ occur in $\uu$, but they are not $H$-equivalent by Lemma~\ref{lem:dvetridyPropP}. It follows that $\mathcal{T}(t)\geq 2$. 

If  $\mathcal{T}(t)=2$, then every element of ${\rm Bext}_H(w)$  is either $H$-equivalent to $0wc$ or $1wc$. It is expressed by the inclusion in Item 1. 

If  the factor $dw\overline{c}$ occurs in $\uu$, then  this factor  should be $H$-equivalent to  $0wc$ or $1wc$. As $w \neq \varepsilon$,  the  morphism $ER$ cannot realize  the  $H$-equivalence; and as $\overline c\not =c$, the identity cannot realize it, either.  Hence, either $w=E(w)$ or $w = R(w)$.  
\end{proof}

\begin{lemma}\label{lem:uzavrenostRorE} Let ${\bf x}_L$ be the sequence associated  in Proposition \ref{pro:RSaLS} with an aperiodic recurrent binary sequence $\uu$. Assume there exists $n_0\in \mathbb N$ such that for all $n\geq n_0$ exactly one equivalence class $t_0 \in \langle \uu\rangle_n$ satisfies $\mathcal{T}(t_0)=2$ and $\mathcal{T}(t) = 1$ for every  $t \in \langle \uu\rangle_n$, $t\neq t_0$. Then 

-- either infinitely many prefixes of ${\bf x}_L$ are $R$-palindromes, in this case the language of ${\bf x}_L$  is closed under $R$ 

-- or infinitely many prefixes of ${\bf x}_L$ are $E$-palindromes, in this case the language of ${\bf x}_L$  is closed under $E$.    
    
\end{lemma}

\begin{proof} Since $\uu$ is recurrent and aperiodic, every factor of $\uu$  is contained in a bispecial factor of $\uu$.   In particular, every  prefix of ${\bf x}_L$ can be prolonged to a prefix  $w$ of ${\bf x}_L$  which is bispecial. It implies that infinitely many prefixes $w$ of ${\bf x}_L$  are bispecial factors.   By Lemma  \ref{lem:simpleObservations1}, $\mathcal{T}(t_0)\geq 2$ for the  equivalence class  $t_0$ containing  $w$. Of course, infinitely many bispecial factors, which are prefixes of ${\bf x}_L$ are longer than $n_0$. Hence by assumption, $\mathcal{T}(t_0) = 2$ and  thus Item 2 of the same lemma forces ${\bf x}_L$ to have infinitely many prefixes which are $R$-palindromes or infinitely many prefixes which are  $E$-palindromes. It confirms the statement. 
\end{proof}

\begin{thm}\label{thm:aperiodicRH} Let $\uu$ be a recurrent binary sequence. Then  ${\mathcal C}_H(n+2) = {\mathcal C}_H(n)+1$ for every $n \in \N$ if and only if  $\uu$ is 
 a complementary symmetric Rote sequence. 
\end{thm}


\begin{proof} 
\noindent $(\Leftarrow)$: By Theorem~\ref{thm:Rote}, the $H$-complexity of any CS Rote sequence satisfies ${\mathcal C}_H(n)=\lfloor \frac{n}{2} \rfloor+1$, which implies ${\mathcal C}_H(n+2)={\mathcal C}_H(n)+1$ for all $n \in \mathbb N$.\\
\noindent $(\Rightarrow)$: As $\uu$ is recurrent, $\mathcal{T}(t) \geq 1$ for every $H$-equivalence class $t$.   
Formula~\eqref{eq:vypocetr(n+2)} implies that for each $n\in \mathbb N$, there exists exactly one class $t_0$ on factors of length $n$ such that $\mathcal{T}(t_0)=2$, for all  other classes $\mathcal{T}(t)=1$.

By Corollary \ref{cor:periodicity}, $\uu$  is aperiodic. Consider the sequence ${\bf x}_L$  associated in Proposition \ref{pro:RSaLS} to the aperiodic sequence $\uu$. Let $ud$ be a left special factor in $\uu$ such that $u$ is of length $n$ and $u$ is not a prefix of 
${\bf x}_L$. By Lemma \ref{lem:simpleObservations1}, the class containing $u$ has $\mathcal{T}(t)\geq 2$, and thus necessarily, $u \sim_H v$, where $vc$ is a~prefix of ${\bf x}_L$ of the same length as $ud$. By Lemma \ref{lem:2veStejneTride}, $ud = \overline{vc} $.  

We will discuss two possibilities. 
\begin{enumerate}
\item $\uu$ contains infinitely many left special factors, which are not prefixes of ${\bf x}_L$: \quad 
Then for every $n\in \mathbb{N}, \ n\geq 1$, there exist precisely  two left special factors:  
one is  a prefix of  ${\bf x}_L$  and the other one is  a prefix of~$\overline{{\bf x}}_L$. Hence by recurrence of $\uu$, we have $\Delta {\mathcal C}(n)= 2$ for all $n\in \mathbb N, \ n\geq 1$. Consequently, $\mathcal{C}(n)=2n$ for every $n\geq 1$ as $\mathcal{C}(1)=2$. 




It remains to show that the language of $\uu$ is closed under the exchange of letters, i.e., under $ER$. 
Since $\uu$ is recurrent,  every factor $w\in \mathcal{L}(\uu)$ is contained in  a left special factor,  in particular, $w$ is a factor of ${\bf x}_L$  or $\overline{{\bf x}}_L$ and therefore $\overline{w}$ is a factor of $\overline{{\bf x}}_L$  or ${\bf x}_L$. Hence, $\overline{w} \in \mathcal{L}({\bf x}_L) \cup  
\mathcal{L}(\overline{{\bf x}}_L)\subset \mathcal{L}(\uu)$. We conclude that $\uu$ is a CS Rote sequence.

\item $\uu$ contains only finitely many left special factors which are not prefixes of ${\bf x}_L$: \quad
For sufficiently large  $n$, there exists only one left special factor, i.e., $\mathcal{C}(n+1) = \mathcal{C}(n)+1$ and $\uu$ is quasi-Sturmian. The unique left special factor of length $n$ is a prefix of ${\bf x}_L$. 
As $\uu$ is recurrent, every $u \in \mathcal{L}(\uu)$ is contained in a sufficiently long bispecial factor. Hence,   $u \in \mathcal{L}({\bf x}_L)$. In other words, $\mathcal{L}(\uu) = \mathcal{L}({\bf x}_L)$.  By Lemma \ref{lem:uzavrenostRorE}, $\mathcal{L}(\uu)$ is closed under $R$ or under $E$. 
Of course, $\mathcal{L}(\uu)$  cannot be  closed simultaneously under $E$ and $R$, as this would imply that each prefix of $\overline{\mathbf{x}}_L$ is left special - contrary to the assumption.

From the fact that $\mathcal{L}(\uu) = \mathcal{L}({\bf x}_L)$ is not closed under $ER$, we deduce that
there are only finitely many prefixes $w$ of ${\bf x}_L$ such that    $\overline w$ belongs to ${\mathcal L}(\mathbf u)$. 

Find the shortest prefix $wc$  of ${\bf x}_L$   with $c\in \{0,1\}$ such that $\overline w\overline c$ is not a factor of $\mathbf u$.  Obviously, $w\not=\varepsilon$. As $\uu$ is recurrent and $\overline{w}$ is a factor of $\uu$,  there exists $a \in \{0,1\}$
 such that $a\overline{w}c$ belong to the language of $\uu$. Let us denote by $t_0$ the equivalence class containing $w$. Necessarily,  $\mathcal{T}(t_0)=2$. Since $0wc, 1wc, a\overline{w}c \in {\rm Bext}_H(w) $  and $0wc\nsim_H 1wc$, we deduce 
\begin{equation}\label{eq:dvaEkviv}
 a\overline{w}c \sim_H  0wc \quad \text{or} \quad \ a\overline{w}c \sim_H  1wc\,. 
\end{equation}
Since $w$ is non-empty and $c$ is the last letter of all three factors    $0wc, 1wc$ and $ a\overline{w}c$,  none of the previous $H$-equivalences can be realized by identity or by $ER$.

Let us distinguish two scenarios.

\begin{enumerate}

\item The language of $\mathbf u$ is closed under $R$.

 By closedness under reversal, $cR(w)0$ and  $cR(w)1$ belong to the language of $\uu$, whereas $\overline{w}\overline{c}$ does not, hence $R(w)\neq \overline{w}$. 

 Therefore,  one of the $H$-equivalences in $\eqref{eq:dvaEkviv}$ is realized by $E$, which implies $a = \overline{c}$ and $\overline{w} = E(w)$.  Note that $E(u) = R(\overline{u})$ for every  $u \in \{0,1\}^*$. Hence, $$a\overline{w}c = \overline{c}\,\overline{w}c =  \overline{c}\,E(w)c =  \overline{c}\,R(\overline{w})c \in \mathcal{L}(\uu).$$  The closedness of the language under  $R$ gives that $c\overline{w} \overline{c}$ belongs to the language - a contradiction, since   $\overline{w} \overline{c}$  is not a factor of $\uu$.  
 
\item The language of $\mathbf u$ is closed under $E$.

 By closedness under $E$, both  $\overline{c}E(w)0$ and  $\overline{c}E(w)1$ belong to the language of $\uu$, whereas $\overline{w}\overline{c}$ does not, hence $E(w)\neq \overline{w}$. 

 Therefore,  one of the $H$-equivalences in $\eqref{eq:dvaEkviv}$ is realized by $R$, which implies $a = c$ and $\overline{w} = R(w)$.  Note that $R(u) = E(\overline{u})$ for every  $u \in \{0,1\}^*$. Hence, $$a\overline{w}c = c\,\overline{w}c = c\,R(w)c =  c\,E(\overline{w})c \in \mathcal{L}(\uu).$$  The closedness of the language under  $E$ gives that $\overline{c}\overline{w}\overline{c}$ belongs to the language - it contradicts the fact  that  $\overline{w}\overline{c} \notin \mathcal{L}(\uu)$.
 
\end{enumerate}
\end{enumerate}

\noindent Both scenarios $(a)$ and $(b)$ in Item 2  of the proof lead to a contradiction, therefore  only the possibility discussed in Item 1 of the proof takes place. 
\end{proof}
Using the proof of Theorem \ref{thm:aperiodicRH} and Lemma \ref{lem:uzavrenostRorE}, we can easily derive the following proposition.

\begin{proposition}\label{pro:quasirote}  Let $\uu$ be a~recurrent binary sequence.  
If  ${\mathcal C}_H(n+2) = {\mathcal C}_H(n)+1$ holds true for every  sufficiently large  $n\in \N$, then there exists a~constant $\kappa$ such that $\uu$  satisfies  one of the following descriptions: 
\begin{enumerate}
    \item $\mathcal{C}(n) = 2n + \kappa$ for  every sufficiently large $n$ and the language of $\uu$ is closed under $H$; 

    \item $\mathcal{C}(n) = n + \kappa$ for  every sufficiently large $n$ and the language of $\uu$ is closed under $R$, but not under $E$; 

    \item  $\mathcal{C}(n) = n + \kappa$ for  every sufficiently large $n$ and the language of $\uu$ is closed under $E$, but not under $R$.  
\end{enumerate}
\end{proposition}

The following examples illustrate that each possibility described in Items 1-3 of Proposition \ref{pro:quasirote}  can occur.

\begin{example} Let us show that the Fibonacci sequence $\mathbf f$ satisfies   $${\mathcal C}_H(n+2) = {\mathcal C}_H(n)+1 \quad \text{ for all $n \in \mathbb{N}$ with the only exception $n =3$}.$$  
The Fibonacci sequence is Sturmian and the language of any Sturmian sequence is closed under reversal.
Let us notice that $\mathbf f$ does not contain $11$, every factor of $\mathbf f$ of length $n\geq 6$ contains $00$, and the only factor of length 5 without $00$ is $01010$ and $E(01010)=10101$ is not a factor of $\mathbf f$.  Hence  $u\sim_H v$ implies that $u=v$ or $u = R(v)$ for all factors of length $\geq 5$. Consequently, for $n\geq 5$, the $H$-complexity coincides with the reflection complexity and thus we can use the result of \cite{AlCaLiShSt2025}. 
The reader can easily check that 

$\mathcal{C}_H(0)=1$,  as  $\mathcal{L}_0({\mathbf f}) = \{\varepsilon\}$

$\mathcal{C}_H(1)=1$,  as  $\mathcal{L}_1({\mathbf f})  = \{0,1\}$

 $\mathcal{C}_H(2)=2$,  as   $\mathcal{L}_2({\mathbf f})  = \{01,10,00\}$

 $\mathcal{C}_H(3)=2$,  as   $\mathcal{L}_3({\mathbf f})  = \{010,101,001,100\}$

  $\mathcal{C}_H(4)=3$,  as   $\mathcal{L}_4({\mathbf f})  = \{0101,1010,1001,0100, 0010\}$  

  $\mathcal{C}_H(n)=r(n) = 1+\lfloor  \frac{n+1}{2}\rfloor$  for all $n\geq 5, \ n \in \N$.
  
\medskip

\noindent This example illustrates case 2 from Proposition~\ref{pro:quasirote} as the factor complexity of $\mathbf f$ satisfies ${\mathcal C}(n)=n+1$ for all $n \in \mathbb N$. 
\end{example}

\begin{example}\label{ex:TMimageFib} Consider $\uu = \pi({\mathbf f})$, where ${\mathbf f}$ is the Fibonacci sequence and 

$$ \pi: \left\{\begin{array}{ll}
    0\mapsto 01\,;\\
    1\mapsto 10\,.
    \end{array}\right. 
$$

It is not difficult to see that if $w$ is an $R$-palindrome, then $\pi(w)$ is an $E$-palindrome. 
On the one hand, as the language of $\mathbf f$ contains infinitely many $R$-palindromes, the language of $\uu$ contains infinitely many $E$-palindromes. 
By uniform recurrence of $\uu$, it means that the language of $\uu$ is closed under $E$.
On the other hand, the language of $\uu$ is not closed under $R$: no factor of $\uu$ of length $\geq 9$ has its reversal in the language of $\uu$.
Consequently, for $n\geq 9$, the $H$-complexity coincides with the $\langle E \rangle$-complexity.
For the lengths $n\geq 9$, 
\begin{itemize}
\item the factors of length $n=2k$ in $\uu$ arise by applying $\pi$ to the factors of length $k$ and $k+1$ in $\mathbf f$ (in the second case we cut the first and the last letter) and there are three $E$-palindromes among them (arising from palindromes of length $k$ and $k+1$ in $\mathbf f$), therefore
$$\mathcal{C}(2k)=(k+1)+(k+2)=2k+3\quad \text{and}\quad \mathcal{C}_H(2k)=\mathcal{C}_{\langle E\rangle}(2k)=k+3\,;$$ 
\item the factors of length $n=2k+1$ in $\uu$ arise by applying $\pi$ to the factors of length $k+1$ in $\mathbf f$ and cutting once the first letter, once the last letter and there are no $E$-palindromes of odd length, therefore
$$\mathcal{C}(2k+1)=2(k+2)=2k+4\quad  \text{and} \quad \mathcal{C}_H(2k+1)=\mathcal{C}_{\langle E\rangle}(2k+1)=k+2\,.$$ 
\end{itemize}
This example illustrates case 3 of Proposition~\ref{pro:quasirote}, as altogether, we have for $n\geq 9$, $${\mathcal C}(n)=n+3\quad \text{and}\quad \mathcal{C}_H(n)=\left\{\begin{array}{ll}
   \frac{n}{2}+3 & \text{for $n$ even} \,;\\
  \frac{n+3}{2} & \text{for $n$ odd} \,.
    \end{array}\right. $$
\end{example} 

\begin{example} Let $\uu$ be a~CS Rote sequence corresponding to the Fibonacci sequence. Consider ${\bf v}= \pi(\uu) $, where  
$$ \pi: \left\{\begin{array}{ll}
    0\mapsto 01110\,;\\
    1\mapsto 10001\,.
    \end{array}\right. 
$$
Since $\uu$ contains infinitely many $R$- and $E$-palindromes and when we apply $\pi$ to a~$\vartheta$-palindrome it is again a~$\vartheta$-palindrome, we get by uniform recurrence that the language of $\vv$ is closed under $H$.

For each factor $u$ of $\vv$ of length $n\geq 5$ there exists a unique factor $v$ and a~unique $i\in\{0,1,2,3,4\}$ such that the factor $u$ occurs at the position $i$ in $\pi(v)$, where $|v|=\lceil \frac{n+i}{5}\rceil$.
Since $\uu$ is the Rote sequence with the factor  complexity $2n$,  the factor complexity of $\vv$  satisfies  ${\mathcal C}(n)=\sum_{i=0}^4 2\lceil \frac{n+i}{5}\rceil=2n+8.$
Moreover, since the CS Rote sequence contains two $R$-palindromes of every length and two $E$-palindromes of every even length, we get, for $n\geq 5$,
$\mathcal{C}_H(n)= \lfloor n/2 \rfloor + 3$.
This example illustrates case 1 of Proposition~\ref{pro:quasirote}.

\end{example}

\section{Open questions} 
Let us formulate the questions that remained unanswered here and which, in our opinion, deserve further investigation.

\begin{enumerate}
    \item   Sequences whose language is closed under $H$ and whose factor complexity is $2n+ \kappa$ appeared naturally in Proposition \ref{pro:quasirote}. It would be very helpful  to find a characterization of such ``CS quasi-Rote sequences'', analogous to the description of quasi-Sturmian sequences in the following theorem. See~\cite{Paul} for the proof for uniformly recurrent sequences and~\cite{Coven} for the general proof.
It has been rediscovered several times, more details may be found in~\cite{Al2000}.
\begin{thm}\label{thm:cassaigne} A sequence $\uu$ over an alphabet $\mathcal{A}$ is quasi-Sturmian if and only if $\uu$ can be written as $\uu = p \varphi({\mathbf v})$, where $p \in \mathcal{A}^*$,  $\mathbf v$ is a Sturmian sequence over $\{a,b\}$ and  $\varphi:  \{a,b\}^* \mapsto \mathcal{A}^*$  is a non-periodic morphism (non-periodic means that $\varphi(ab)\neq \varphi(ba)$). 
\end{thm}
Note that if $\uu$ is recurrent and  quasi-Sturmian, then $\uu$ is uniformly recurrent. Let us explain why this is so: 

Since any Sturmian word $\mathbf v$ is uniformly recurrent,   the language of $\varphi(\mathbf v)$ is uniformly recurrent, too. As $\uu = p \varphi({\mathbf v})$ is recurrent, each of its factors $u$ occurs in $\uu $  infinitely many times. With a finite number of exceptions, all occurrences of $u$ in $\uu$ are found in $\varphi(\mathbf v)$. Hence, there are only finitely many return words to $u$ in $\uu$. It means that $\uu$ is uniformly recurrent, too.

\medskip

Are the ``CS quasi-Rote sequences'' uniformly recurrent, too?

   \item The aperiodic sequences with the minimum reflection complexity were characterized in Theorem~\ref{thm:Rminimum} without the assumption of recurrence of the sequence. 
   By contrast, in Theorem~\ref{thm:aperiodicRH} we characterize the aperiodic sequences with the minimum $H$-complexity under the assumption of recurrence. Can this assumption be dropped?
    \item  What can be said about the $H$-complexity of a purely morphic sequence whose language is closed under $H$? Note that every such sequence is fixed by a morphism from a class defined in  \cite{AmMaPe}. 
    \item  Fulfillment of the inequality   \ 
    $ \mathcal{C}_G(n+2)\geq \mathcal{C}_G(n)+1$ for every positive integer $n$ characterizes aperiodic sequences if the group $G$ possesses Property $\textswab{P}$,  as shown in Proposition \ref{pro:neperiodicity}.   Can aperiodicity be characterized even if the group does not possess this property?

\item The group $H$ and the group generated by the reflection $R$  possess Property $\textswab{P}$. Aperiodic sequences with the minimum generalized reflection complexity are known for these two groups; see Theorems  \ref{thm:Rminimum} and \ref{thm:aperiodicRH}. Which aperiodic sequences have the minimum $G$-complexity for other groups with Property $\textswab{P}$?

\end{enumerate}

\end{document}